\documentclass[preprint,12pt]{elsarticle1}

\usepackage{subfig}
\usepackage{graphicx}
\usepackage{caption,color}   
\usepackage{amssymb}
\usepackage{amsthm}
\usepackage{amsmath}
\usepackage{epic}
\usepackage{setspace}
\usepackage{float}
\usepackage{multirow}
\newtheorem{thm}{Theorem}[section]
\newtheorem{cor}[thm]{Corollary}

\newtheorem{lem}[thm]{Lemma}

\numberwithin{equation}{section}
\journal{}

\begin{document}
\begin{spacing}{1.15}
\begin{frontmatter}
\title{Lexicographical ordering of hypergraphs via spectral moments}

\author{Hong Zhou}
\author{Changjiang Bu}\ead{buchangjiang@hrbeu.edu.cn}
\address{College of Mathematical Sciences, Harbin Engineering University, Harbin 150001, PR China}

\begin{abstract}
The lexicographical ordering  of hypergraphs via spectral moments is called the $S$-order of hypergraphs.
In this paper, the $S$-order of hypergraphs is investigated.
We characterize the first and last hypergraphs in an $S$-order of all uniform hypertrees and all linear unicyclic uniform
hypergraphs with given girth, respectively. And we give the last hypergraph in an $S$-order of  all linear unicyclic uniform  hypergraphs.


\end{abstract}

\begin{keyword}
hypergraph, spectral moment, adjacency tensor\\
\emph{AMS classification (2020):}
05C65, 15A18
\end{keyword}
\end{frontmatter}

\section{Introduction}
Let $G$ be a simple undirected graph with $n$ vertices and $A$ be the adjacency matrix of $G$.
The \textit{$d$th order spectral moment} of $G$ is the sum of $d$ powers of all the eigenvalues of $A$,  denoted by $\mathrm{S}_{d}(G)$ \cite{1980Spectra}.
For two graphs $G_{1}, G_{2}$ with $n$ vertices, if $\mathrm{S}_{i}(G_{1})=\mathrm{S}_{i}(G_{2})$ for $i=0,1,2,\ldots,n-1$, then adjacency matrices of $G_{1}$ and $G_{2}$ have the same spectrum.  Therefore, $\mathrm{S}_{i}(G_{1})=\mathrm{S}_{i}(G_{2})$ for $i=0,1,2,\ldots$.
We write $G_{1}\prec_{s} G_{2}$ ($G_{1}$ comes before $G_{2}$ in an $S$-order) if there exists a $k\in\{1,2,\ldots,n-1\}$ such that $\mathrm{S}_{i}(G_{1})=\mathrm{S}_{i}(G_{2})$ for $i=0,1,2,\ldots,k-1$ and
$\mathrm{S}_{k}(G_{1})<\mathrm{S}_{k}(G_{2})$. We write $G_{1}=_{s} G_{2}$, if $\mathrm{S}_{i}(G_{1})=\mathrm{S}_{i}(G_{2})$ for $i=0,1,2,\ldots,n-1$.


In 1987, Cvetkovi\'{c} and Rowlinson \cite{cvetkovic1987spectra} characterized the first and last graphs in an $S$-order of all trees and all unicyclic
graphs with given girth, respectively.
Other works on the $S$-order of graphs can be referred to \cite{WU20101707, PAN20111265, CHENG20121123,CHENG2012858,SHUCHAO2013ON,10.1216/RMJ-2016-46-1-261}.
The $S$-order of graphs had been used in producing graph catalogues \cite{Drago1984A}.

In this paper, the $S$-order of hypergraphs is defined.
We characterize the first and last hypergraphs in an $S$-order of all uniform hypertrees and all linear unicyclic uniform
hypergraphs with given girth, respectively.  And we give the last hypergraph in an $S$-order of  all linear unicyclic uniform  hypergraphs.


Next, we introduce some notations and concepts for tensors and hypergraphs. For a positive integer $n$, let $[n]=\{1,2,\ldots,n\}$. An $m$-order $n$-dimension complex \textit{tensor}
$
\mathcal{A}=\left( {a_{i_{1}\cdots i_{m}} } \right)
$
is a multidimensional array with $n^m$ entries on complex number field $\mathbb{C}$, where $i_{j}\in [n], j=1,\ldots,m$.

Let $\mathbb{C}^{n}$ be the set of $n$-dimension complex vectors and $\mathbb{C}^{[m,n]}$ be the set of $m$-order $n$-dimension complex tensors.
For $x=\left({x_1 ,\ldots ,x_n}\right)^\mathrm{T}\in\mathbb{C}^n$,  $\mathcal{A}x^{m-1}$ is a vector in $\mathbb{C}^n$ whose $i$th component is
\begin{align*}
(\mathcal{A}x^{m-1})_i=\sum\limits_{i_{2},\ldots,i_{m}=1}^{n}a_{ii_{2}\cdots i_{m}}x_{i_{2}}\cdots x_{i_{m}}.
\end{align*}
A number $\lambda\in\mathbb{C}$ is called an \textit{eigenvalue} of $\mathcal{A}$ if there exists a nonzero vector $x\in\mathbb{C}^n$ such that
$$\mathcal{A}x^{m-1}=\lambda x^{[m-1]},$$
where $x^{\left[ {m - 1} \right]}  = \left( {x_1^{m - 1} ,\ldots,x_n^{m - 1} } \right)^\mathrm{T}$.
The number of eigenvalues of $\mathcal{A}$ is $n(m-1)^{n-1}$ \cite{qi2005eigenvalues,lim2005singular}.

A hypergraph $\mathcal{H}=(V(\mathcal{H}), E(\mathcal{H}))$ is called \textit{$m$-uniform} if $|e|=m\geq2$ for all $e\in E(\mathcal{H})$. For an $m$-uniform hypergraph $\mathcal{H}$ with $n$ vertices, its  \textit{adjacency tensor} is the order $m$ dimension $n$ tensor
$\mathcal{A}_\mathcal{H}=(a_{i_{1}i_{2}\cdots i_{m}})$, where
\begin{equation*}
a_{i_{1}i_{2}\cdots i_{m}}=\begin{cases}
\frac{1}{(m-1)!},& \text{if } \{i_{1},i_{2},\ldots,i_{m}\}\in E(\mathcal{H}),\notag \\
0,& \text{otherwise}.
\end{cases}
\end{equation*}
Clearly, $\mathcal{A}_\mathcal{H}$ is the adjacency matrix of $\mathcal{H}$ when $\mathcal{H}$ is $2$-uniform \cite{cooper2012spectra}.
The \textit{degree} of a vertex $v$ of $\mathcal{H}$ is the number of edges containing the vertex, denoted by $d_{\mathcal{H}}(v)$ or $d_{v}$.
A vertex of $\mathcal{H}$ is called a \textit{core vertex} if it has degree one.
An edge $e$ of $\mathcal{H}$ is called a \textit{pendent edge} if it contains $|e|-1$ core vertices.
Sometimes a core vertex in a pendent edge is also called a \textit{pendent vertex}.
The \textit{girth} of $\mathcal{H}$ is the minimum length of the hypercycles of $\mathcal{H}$, denoted by $g(\mathcal{H})$.
$\mathcal{H}$ is called \textit{linear} if any two different edges intersect into at most one vertex.
The \textit{$m$-power hypergraph} $G^{(m)}$ is the $m$-uniform hypergraph which obtained by adding $m-2$ vertices with degree one to each edge of the graph $G$.

In 2005, the concept of eigenvalues of tensors was proposed by Qi \cite{qi2005eigenvalues} and Lim \cite{lim2005singular}, independently.
The eigenvalues of tensors and related problems are important research topics of spectral hypergraph theories \cite{cooper2,doi:10.1137/21M1404740,2017Tensor,clark2021harary}, especially the trace of tensors \cite{clark2021harary,2011Analogue,hu2013determinants,shao2015some,doi:10.1080/03081087.2021.1953431}.



Morozov and Shakirov  gave an expression of the $d$th order trace $\mathrm{Tr}_{d}(\mathcal{A})$ of a tensor $\mathcal{A}$ \cite{2011Analogue}. Hu et al. proved that $\mathrm{Tr}_{d}(\mathcal{A})$ is equal to the sum of $d$ powers of all eigenvalues of $\mathcal{A}$ \cite{hu2013determinants}. For a uniform hypergraph $\mathcal{H}$, the sum of $d$ powers of all eigenvalues of $\mathcal{A}_\mathcal{H}$ is called the \textit{$d$th order  spectral moment} of $\mathcal{H}$, denoted by $\mathrm{S}_{d}(\mathcal{H})$. Then
$\mathrm{Tr}_{d}(\mathcal{\mathcal{A}_\mathcal{H}})=\mathrm{S}_{d}(\mathcal{H})$. Shao et al. established some formulas for the $d$th order trace of tensors in terms of some graph parameters \cite{shao2015some}.
Clark and Cooper expressed the spectral moments of
hypergraphs by the number of Veblen multi-hypergraphs and used this result to give the ``Harary-Sachs'' coefficient theorem for hypergraphs \cite{clark2021harary}.
Chen et al. gave a formula for the spectral moment of a hypertree in terms of the number of some sub-hypertrees \cite{doi:10.1080/03081087.2021.1953431}.

This paper is organized as follows. In Section 2, the $S$-order of hypergraphs is defined. We introduce $4$ operations of moving edges on hypergraphs and give changes of the Zagreb index after operations of moving edges. In Section 3, we give the first and last hypergraphs in an $S$-order of all uniform hypertrees. In Section 4, the expressions of $2m$th and $3m$th order spectral moments of linear unicyclic $m$-uniform  hypergraphs are
obtained in terms of the number of sub-hypergraphs. We characterize the first and last hypergraphs in an $S$-order of  all linear unicyclic uniform  hypergraphs with given girth. And we give the last hypergraph in an $S$-order of  all linear unicyclic uniform  hypergraphs.

\section{Preliminaries}

For two $m$-uniform hypergaphs $\mathcal{H}_{1},\mathcal{H}_{2}$ with $n$ vertices,
if $\mathrm{S}_{i}(\mathcal{H}_{1})=\mathrm{S}_{i}(\mathcal{H}_{2})$ for $i=0,1,2,\ldots,n(m-1)^{n-1}-1$, then adjacency tensors of $\mathcal{H}_{1}$ and $\mathcal{H}_{2}$  have the same spectrum. Therefore, $\mathrm{S}_{i}(\mathcal{H}_{1})=\mathrm{S}_{i}(\mathcal{H}_{2})$ for $i=0,1,2,\ldots$.
We write $\mathcal{H}_{1}\prec_{s} \mathcal{H}_{2}$ ($\mathcal{H}_{1}$ comes before $\mathcal{H}_{2}$ in an $S$-order) if there exists a $k\in\{1,2,\ldots,n(m-1)^{n-1}-1\}$ such that $\mathrm{S}_{i}(\mathcal{H}_{1})=\mathrm{S}_{i}(\mathcal{H}_{2})$ for $i=0,1,2,\ldots,k-1$ and
$\mathrm{S}_{k}(\mathcal{H}_{1})<\mathrm{S}_{k}(\mathcal{H}_{2})$. 
We write $\mathcal{H}_{1}=_{s} \mathcal{H}_{2}$ if $\mathrm{S}_{i}(\mathcal{H}_{1})=\mathrm{S}_{i}(\mathcal{H}_{2})$ for $i=0,1,2,\ldots,n(m-1)^{n-1}-1$.
In this paper, $\mathrm{S}_i(\mathcal{H})$ is also written $\mathrm{S}_i, i=0,1,2,\ldots$. Let $\textbf{H}_{1}$ and $\textbf{H}_{2}$ be two sets of hypergraphs. We write $\textbf{H}_{1}\prec_{s}\textbf{H}_{2}$ ($\textbf{H}_{1}$  comes before $\textbf{H}_{2}$ in an $S$-order) if $\mathcal{H}_{1}\prec_{s}\mathcal{H}_{2}$ for each $\mathcal{H}_{1}\in \textbf{H}_{1}$ and each $\mathcal{H}_{2} \in \textbf{H}_{2}$.

For an $m$-uniform hypergraph $\mathcal{H}$ with $n$ vertices, let $\mathrm{S}_{0}(\mathcal{H})=n(m-1)^{n-1}.$
In \cite{cooper2012spectra}, the $d$th order traces of the adjacency tensor of an $m$-uniform hypergraph were given for $d=1,2,\ldots,m$.
\begin{lem}\cite{cooper2012spectra}\label{L2}
Let $\mathcal{H}$ be an $m$-uniform hypergraph with $n$ vertices and $q$ edges. Then

(1) $\mathrm{Tr}_d(\mathcal{A}_\mathcal{H})=0$ for $d=1,2,\ldots, m-1$;

(2) $\mathrm{Tr}_{m}(\mathcal{A}_\mathcal{H})=qm^{m-1}(m-1)^{n-m}$.

\end{lem}
Next, we introduce $4$ operations of moving edges on hypergraphs and give changes of the Zagreb index after operations of moving edges.
The sum of the squares of the degrees of all vertices of a hypergraph $\mathcal{H}$ is called the \textit{Zagreb index} of $\mathcal{H}$, denoted by $M(\mathcal{H})$ \cite{Kau2020Energies}.
Let $E'\subseteq E(\mathcal{H})$, we denote by $\mathcal{H}-E'$ the sub-hypergraph of $\mathcal{H}$ obtained by deleting the edges of $E'$.

$\textbf{Transformation 1}$: Let $e=\{u,v,v_{1},v_{2},\ldots,v_{m-2}\}$ be an edge of an $m$-uniform hypergraph $\mathcal{H}$, $e_{1},e_{2},\ldots,e_{t}$ be the pendent edges incident with $u$, where $t\geq 1$, $d_{\mathcal{H}}(u)=t+1$ and $d_{\mathcal{H}}(v)\geq 2$. Write $e_{i}^{'}=(e_{i}\setminus \{u\})\bigcup\{v\}$. Let $\mathcal{H}^{'}=\mathcal{H}-\{e_{1},\ldots,e_{t}\}+\{e'_{1},\ldots,e'_{t}\}$.
\begin{lem}\label{sp1}
Let $\mathcal{H}'$ be obtained from $\mathcal{H}$ by transformation 1. Then $M(\mathcal{H}')>M(\mathcal{H})$.
\end{lem}
\begin{proof}
By the definition of the Zagreb index, we have
\begin{align*}
M(\mathcal{H}')-M(\mathcal{H})&=d^{2}_{\mathcal{H}'}(v)-d^{2}_{\mathcal{H}}(v)+d^{2}_{\mathcal{H}'}(u)-d^{2}_{\mathcal{H}}(u)\\
&=(d_{\mathcal{H}}(v)+t)^{2}-d_{\mathcal{H}}^{2}(v)+1-(t+1)^{2}\\
&=2t(d_{\mathcal{H}}(v)-1)>0.
\end{align*}
\end{proof}
$\textbf{Transformation 2}$: Let $u$ and $v$ be two vertices in a uniform hypergraph $\mathcal{H}$, $e_{1},e_{2},\ldots,e_{r}$ be the pendent edges incident with $u$ and $e_{r+1},e_{r+2},\ldots,e_{r+t}$ be the pendent edges incident with $v$, where $r\geq 1$ and $t\geq 1$.  Write $e_{i}^{'}=(e_{i}\setminus \{u\})\bigcup\{v\}, i\in[r]$, $e_{i}^{'}=(e_{i}\setminus \{v\})\bigcup\{u\}, i=r+1,\ldots,r+t$. If $d_{\mathcal{H}}(v)\geq d_{\mathcal{H}}(u)$, let  $\mathcal{H}'=\mathcal{H}-\{e_{1},\ldots,e_{r}\}+\{e'_{1},\ldots,e'_{r}\}$. If $d_{\mathcal{H}}(v)<d_{\mathcal{H}}(u)$, let $\mathcal{H}'=\mathcal{H}-\{e_{r+1},\ldots,e_{r+t}\}+\{e'_{r+1},\ldots,e'_{r+t}\}$.
\begin{lem}\label{sp2}
Let $\mathcal{H}'$ be obtained from $\mathcal{H}$ by transformation 2. Then $M(\mathcal{H}')>M(\mathcal{H})$.
\end{lem}
\begin{proof}
By the definition of the Zagreb index, if $d_{\mathcal{H}}(v)\geq d_{\mathcal{H}}(u)$, we have
\begin{align*}
M(\mathcal{H}')-M(\mathcal{H})&=d^{2}_{\mathcal{H}'}(v)-d^{2}_{\mathcal{H}}(v)+d^{2}_{\mathcal{H}'}(u)-d^{2}_{\mathcal{H}}(u)\\
&=(d_{\mathcal{H}}(v)+r)^{2}-d_{\mathcal{H}}^{2}(v)+(d_{\mathcal{H}}(u)-r)^{2}-d_{\mathcal{H}}^{2}(u)\\
&=2r(r+d_{\mathcal{H}}(v)-d_{\mathcal{H}}(u))>0.
\end{align*}
If $d_{\mathcal{H}}(v)<d_{\mathcal{H}}(u)$, we have
\begin{align*}
M(\mathcal{H}')-M(\mathcal{H})&=d^{2}_{\mathcal{H}'}(v)-d^{2}_{\mathcal{H}}(v)+d^{2}_{\mathcal{H}'}(u)-d^{2}_{\mathcal{H}}(u)\\
&=(d_{\mathcal{H}}(v)-t)^{2}-d_{\mathcal{H}}^{2}(v)+(d_{\mathcal{H}}(u)+t)^{2}-d_{\mathcal{H}}^{2}(u)\\
&=2t(t+d_{\mathcal{H}}(u)-d_{\mathcal{H}}(v))>0.
\end{align*}
\end{proof}

The $m$-uniform hypertree with a maximum degree of less than or equal to $2$ is called the \textit{binary $m$-uniform hypertree}.
For two vertices $u,v$ of  an $m$-uniform  hypergraph $\mathcal{H}$, the \textit{distance} between $u$ and $v$  is the length of a shortest  path from $u$ to $v$, denoted by $d_{\mathcal{H}}(u, v)$ \cite{LIN2016564}. Let $d_{\mathcal{H}}(u, u)=0$. Let $\mathcal{H}_{0},\mathcal{H}_{1},\ldots,\mathcal{H}_{p}$ be pairwise disjoint connected hypergraphs with $v_{1},\ldots,v_{p}\in V(\mathcal{H}_{0})$ and $u_{i}\in V(\mathcal{H}_{i})$ for each $i\in[p]$, where $p\geq 1$.
Denote by $\mathcal{H}_{0}(v_{1},\ldots,v_{p})\bigodot(\mathcal{H}_{1}(u_{1}),\ldots,\mathcal{H}_{p}(u_{p}))$ the hypergraph
obtained from $\mathcal{H}_{0}$ by attaching $\mathcal{H}_{1},\ldots,\mathcal{H}_{p}$ to $\mathcal{H}_{0}$ with $u_{i}$ identified with $v_{i}$ for each
$i\in[p]$ \cite{FAN202389}.
Let $P_{q}$ be a path of length $q$.

$\textbf{Transformation 3}$: Let $\mathcal{H}\neq P_{0}^{(m)}$ be an $m$-uniform connected  hypergraph with $u\in{V(\mathcal{H})}$. Let $\mathcal{T}$ be a binary $m$-uniform hypertree with $v_{k}, v_{n}, u_{1}, u_{2} \in V(\mathcal{T})$ and $e_{k}, e_{k+1}\in E(\mathcal{T})$ such that $d_{\mathcal{T}}(v_{k})=2$, $v_{k}, u_{1}\in e_{k}, v_{k}, u_{2}\in e_{k+1}$, $u_{1},  u_{2}\neq v_{k}$, $v_{n}$ be a pendent vertex and $d_{\mathcal{T}}(u_{1},v_{n})>d_{\mathcal{T}}(u_{2},v_{n})$. Let $\mathcal{H}_{1}=\mathcal{H}(u)\bigodot\mathcal{T}(v_{k})$.  $\mathcal{H}_{2}$ is obtained from $\mathcal{H}_{1}$ by deleting $e_{k}$ and adding $(e_{k}\setminus \{v_{k}\})\bigcup\{v_{n}\}$.
\begin{lem}\label{sp4}
Let $\mathcal{H}_{2}$ be obtained from $\mathcal{H}_{1}$ by transformation 3. Then $M(\mathcal{H}_{1})>M(\mathcal{H}_{2})$.
\end{lem}
\begin{proof}
By the definition of the Zagreb index, we have
\begin{align*}
M(\mathcal{H}_{1})-M(\mathcal{H}_{2})&=d^{2}_{\mathcal{H}_{1}}(v_{k})+d^{2}_{\mathcal{H}_{1}}(v_{n})-d^{2}_{\mathcal{H}_{2}}(v_{k})-d^{2}_{\mathcal{H}_{2}}(v_{n})\\
&=(d_{\mathcal{H}}(u)+2)^{2}+1-(d_{\mathcal{H}}(u)+1)^{2}-4\\
&=2d_{\mathcal{H}}(u)>0.
\end{align*}
\end{proof}
$\textbf{Transformation 4}$:
Let $\mathcal{H}$ be an $m$-uniform connected  hypergraph with $u, v\in{V(\mathcal{H})}$  such that $u\neq v$, $d_{\mathcal{H}}(u)>1$ and $d_{\mathcal{H}}(u)\geq d_{\mathcal{H}}(v)$.
Let $\mathcal{T}_{1},\mathcal{T}_{2}$ be two binary $m$-uniform hypertrees, where $|E(\mathcal{T}_{1})|>0$.
$\mathcal{H}_{1}$
denotes the hypergraph that results from identifying $u$ with the pendent vertex $u_{0}\in e_{0}$ of $\mathcal{T}_{1}$ and identifying $v$ with the pendent vertex $v_{0}$ of $\mathcal{T}_{2}$. Suppose that  $v_{t}\in V(\mathcal{T}_{2})$ is a pendent vertex of $\mathcal{H}_{1}$, let $\mathcal{H}_{2}$ be obtained from $\mathcal{H}_{1}$ by deleting $e_{0}$ and adding $(e_{0}\setminus\{u\})\bigcup\{v_{t}\}$.


\begin{lem}\label{sp5}
Let $\mathcal{H}_{2}$ be obtained from $\mathcal{H}_{1}$ by transformation 4.

(1). If $|E(\mathcal{T}_{2})|>0$, then $M(\mathcal{H}_{1})>M(\mathcal{H}_{2})$;

(2). If $|E(\mathcal{T}_{2})|=0, d_{\mathcal{H}}(u)> d_{\mathcal{H}}(v)$, then $M(\mathcal{H}_{1})>M(\mathcal{H}_{2})$.
\end{lem}
\begin{proof}
By the definition of the Zagreb index,
if $|E(\mathcal{T}_{2})|>0$, we have
\begin{align*}
M(\mathcal{H}_{1})-M(\mathcal{H}_{2})&=d^{2}_{\mathcal{H}_{1}}(u)+d^{2}_{\mathcal{H}_{1}}(v_{t})-d^{2}_{\mathcal{H}_{2}}(u)-d^{2}_{\mathcal{H}_{2}}(v_{t})\\
&=(d_{\mathcal{H}}(u)+1)^{2}+1-d_{\mathcal{H}}^{2}(u)-4\\
&=2d_{\mathcal{H}}(u)-2>0.
\end{align*}
If $|E(\mathcal{T}_{2})|=0$, $d_{\mathcal{H}}(u)> d_{\mathcal{H}}(v)$, we have
\begin{align*}
M(\mathcal{H}_{1})-M(\mathcal{H}_{2})&=d^{2}_{\mathcal{H}_{1}}(u)+d^{2}_{\mathcal{H}_{1}}(v_{t})-d^{2}_{\mathcal{H}_{2}}(u)-d^{2}_{\mathcal{H}_{2}}(v_{t})\\
&=(d_{\mathcal{H}}(u)+1)^{2}+d_{\mathcal{H}}^{2}(v)-d_{\mathcal{H}}^{2}(u)-(d_{\mathcal{H}}(v)+1)^{2}\\
&=2d_{\mathcal{H}}(u)-2d_{\mathcal{H}}(v)>0.
\end{align*}
\end{proof}


\section{The $S$-order in hypertrees}

In this section, we give the first and last hypergraphs in an $S$-order of all uniform hypertrees.

In \cite{doi:10.1080/03081087.2021.1953431}, the first $3k$th order spectral moments of uniform hypertrees were given. Let $N_{\mathcal{H}}(\widehat{\mathcal{H}})$ be the number of sub-hypergraphs of $\mathcal{H}$ isomorphic to $\widehat{\mathcal{H}}$ and $S_{q}$ be a star with $q$ edges.
\begin{lem}\cite{doi:10.1080/03081087.2021.1953431}\label{w1}
Let $\mathcal{T}=(V(\mathcal{T}),E(\mathcal{T}))$ be an $m$-uniform hypertree. Then
\begin{align*}
{\mathrm{S}_{m}}(\mathcal{T})&= {m^{m - 1}}{(m - 1)^{(|E(\mathcal{T})| - 1)(m - 1)}}{N_{\mathcal{T}}(P^{(m)}_1)} ,\\
\mathrm{S}_{2m}(\mathcal{T})&=m^{m-1}(m-1)^{(|E(\mathcal{T})|-1)(m-1)}N_{\mathcal{T}}(P_{1}^{(m)})+2m^{2m-3}(m-1)^{(|E(\mathcal{T})|-2)(m-1)}N_{\mathcal{T}}(P_{2}^{(m)}),\\
\mathrm{S}_{3m}(\mathcal{T})&=m^{m-1}(m-1)^{(|E(\mathcal{T})|-1)(m-1)}N_{\mathcal{T}}(P_{1}^{(m)})+6m^{2m-3}(m-1)^{(|E(\mathcal{T})|-2)(m-1)}N_{\mathcal{T}}(P_{2}^{(m)})\\&+3m^{3m-5}(m-1)^{(|E(\mathcal{T})|-3)(m-1)}N_{\mathcal{T}}(P_{3}^{(m)})+6m^{3m-5}(m-1)^{(|E(\mathcal{T})|-3)(m-1)}N_{\mathcal{T}}(S_{3}^{(m)}),\\
\mathrm{S}_{d}(\mathcal{T})&=0, \text{~for~} d=1,\ldots, m-1, m+1,\ldots, 2m-1, 2m+1,\ldots, 3m-1.\\
\end{align*}
\end{lem}

Let $\textbf{T}_{q}$ be the set of all $m$-uniform hypertrees with $q$ edges. The following theorem gives the last hypergraph in an $S$-order of all $m$-uniform hypertrees.
\begin{thm}\label{sp6}
In an $S$-order of $\textbf{T}_{q}$, the last hypergraph is the hyperstar $S_{q}^{(m)}$.
\end{thm}
\begin{proof}
Since in all $m$-uniform hypertrees with $q$ edges the spectral moments $\mathrm{S}_{0},\mathrm{S}_{1},\\\ldots,\mathrm{S}_{2m-1}$ are the same, the first significant spectral moment is the $2m$th. By Lemma \ref{w1}, $\mathrm{S}_{2m}$ is determined by the number of $P_{2}^{(m)}$. The number of vertices of $m$-uniform hypertrees with $q$ edges is $qm-q+1$. For any hypertree $\mathcal{T}$ in $\textbf{T}_{q}$, we have
$$
N_{\mathcal{T}}(P_{2}^{(m)})=\sum\limits_{i=1}^{qm-q+1}{d_{i}\choose2}=\frac{1}{2}\sum\limits_{i=1}^{qm-q+1}d_{i}^{2}-\frac{qm}{2}=\frac{1}{2}M(\mathcal{T})-\frac{qm}{2},
$$
where $d_{1}+d_{2}+\cdots+d_{qm-q+1}=mq$.

Repeating transformation 1, any $m$-uniform hypertree with $q$ edges can changed into $S_{q}^{(m)}$.
And by Lemma \ref{sp1}, each application of transformation 1 strictly increases the Zagreb index.
Therefore, in an $S$-order of $\textbf{T}_{q}$, the last hypergraph is the hyperstar $S_{q}^{(m)}$.
\end{proof}
Let $\textbf{T}$ be the set of all binary $m$-uniform hypertrees with $q$ edges. We characterize the first few hypergraphs
in the $S$-order of all $m$-uniform hypertrees.
\begin{thm}\label{zsy1}
$\textbf{T}\prec_{s}\textbf{T}_{q}\setminus \textbf{T}$.
\end{thm}
\begin{proof}

As in the proof of Theorem \ref{sp6} we pay attention to the Zagreb index.
Repeating transformation 3, any $m$-uniform hypertree with $q$ edges can changed into a binary $m$-uniform hypertree with $q$ edges. And from Lemma \ref{sp4},
each application of transformation 3 strictly decreases the Zagreb index.
Hence, $\textbf{T}\prec_{s}\textbf{T}_{q}\setminus \textbf{T}$.
\end{proof}

Let $P_{3}(\mathcal{H})$ be the set of all sub-hyperpaths length $3$ of an $m$-uniform hypergraph $\mathcal{H}$.
\begin{lem}\label{z1}
Let $e=\{u,v,w_{1},\ldots,w_{m-2}\}$ be an edge and $\mathcal{H}_{1},\ldots,\mathcal{H}_{p}$ be pairwise disjoint  connected $m$-uniform hypergraphs with $\mathcal{H}_{i}\neq P_{0}^{(m)}$ and $\widetilde{w}_{i}\in V(\mathcal{H}_{i})$ for each $i\in[p]$, where $m\geq 3$, $1\leq p\leq m-2$. Let $\mathcal{H}=e(w_{1},\ldots,w_{p})\bigodot(\mathcal{H}_{1}(\widetilde{w}_{1}),\ldots,\mathcal{H}_{p}(\widetilde{w}_{p}))$. Let $\mathcal{H}_{r,s}^{e}=\mathcal{H}(u,v)\bigodot(P_{r}^{(m)}(\widetilde{u}),P_{s}^{(m)}(\widetilde{v}))$, where $\widetilde{u},\widetilde{v}$ are respectively the pendent vertices of $P_{r}^{(m)}$ and $P_{s}^{(m)}$. If $r\geq s\geq 1$, then
$$
N_{\mathcal{H}_{r,s}^{e}}(P_{3}^{(m)})> N_{\mathcal{H}_{r+s,0}^{e}}(P_{3}^{(m)}).
$$
\end{lem}
\begin{proof}
Since $p\geq1 $, let $e_{1}\in E(\mathcal{H}_{1})$ be an edge incident with $\widetilde{w}_{1}$.
Let $e_{2}\in E(P_{r}^{(m)})$ be an edge incident with $\widetilde{u}$ and $e_{3}\in E(P_{s}^{(m)})$ be an edge incident with $\widetilde{v}$. We have $P_{3}(\mathcal{H}_{r,0}^{e})\subseteq P_{3}(\mathcal{H}_{r,s}^{e})$ and $P_{3}(\mathcal{H}_{r,0}^{e})\subseteq P_{3}(\mathcal{H}_{r+s,0}^{e})$. For a hyperpath $\mathcal{P}_{1}$ with $E(\mathcal{P}_{1})=\{e,e',e''\}$, $\mathcal{P}_{1}$ is also written $ee'e''$ in this paper.

If $s=1$,
there are hyperpaths $e_{2}ee_{3},e_{3}ee_{1}$
in $P_{3}(\mathcal{H}_{r,1}^{e})$ and not in $P_{3}(\mathcal{H}_{r,0}^{e})$. Since $p\geq1$,
$N_{\mathcal{H}_{r,1}^{e}}(P_{3}^{(m)})- N_{\mathcal{H}_{r,0}^{e}}(P_{3}^{(m)})\geq2.$
There is only one hyperpath $\mathcal{P}$ in $P_{3}(\mathcal{H}_{r+1,0}^{e})$ and not in $P_{3}(\mathcal{H}_{r,0}^{e})$. And the edges of $\mathcal{P}$ are not in $E(\mathcal{H}_{i}), i=1,2,\ldots,p$. We have $N_{\mathcal{H}_{r+1,0}^{e}}(P_{3}^{(m)})- N_{\mathcal{H}_{r,0}^{e}}(P_{3}^{(m)})=1.$ So, $N_{\mathcal{H}_{r,1}^{e}}(P_{3}^{(m)})> N_{\mathcal{H}_{r+1,0}^{e}}(P_{3}^{(m)})$.

If $s=2$, let $e_{4}\neq e_{3}\in E(P_{s}^{(m)})$. There are hyperpaths $e_{2}ee_{3},e_{3}ee_{1}, ee_{3}e_{4}$  in $P_{3}(\mathcal{H}_{r,2}^{e})$ and not in $P_{3}(\mathcal{H}_{r,0}^{e})$. Since $p\geq1$,
$N_{\mathcal{H}_{r,2}^{e}}(P_{3}^{(m)})- N_{\mathcal{H}_{r,0}^{e}}(P_{3}^{(m)})\geq3.$
There are only two hyperpaths $\mathcal{P}'$, $\mathcal{P}''$ in $P_{3}(\mathcal{H}_{r+2,0}^{e})$ and not in $P_{3}(\mathcal{H}_{r,0}^{e})$. And the edges of $\mathcal{P}'$ and $\mathcal{P}''$ are not in $E(\mathcal{H}_{i}), i=1,2,\ldots,p$. We have $N_{\mathcal{H}_{r+2,0}^{e}}(P_{3}^{(m)})- N_{\mathcal{H}_{r,0}^{e}}(P_{3}^{(m)})=2.$  So, $N_{\mathcal{H}_{r,2}^{e}}(P_{3}^{(m)})> N_{\mathcal{H}_{r+2,0}^{e}}(P_{3}^{(m)})$.

If $s>2$, similar to $s=2$, there are hyperpaths $e_{2}ee_{3},e_{3}ee_{1}, ee_{3}e_{4}$  in $P_{3}(\mathcal{H}_{r,s}^{e})$ and not in $P_{3}(\mathcal{H}_{r,0}^{e})$. For an $m$-uniform hyperpath with $q~ (q>2)$ edges, the number of the sub-hyperpaths with $3$ edges is $q-2$. Since  $p\geq1 $, $$N_{\mathcal{H}_{r,s}^{e}}(P_{3}^{(m)})- N_{\mathcal{H}_{r,0}^{e}}(P_{3}^{(m)})\geq 3+s-2=s+1.$$ Since $r\geq s>2$, there are only $s$ hyperpaths  in $P_{3}(\mathcal{H}_{r+s,0}^{e})$ and not in $P_{3}(\mathcal{H}_{r,0}^{e})$. We have $N_{\mathcal{H}_{r+s,0}^{e}}(P_{3}^{(m)})- N_{\mathcal{H}_{r,0}^{e}}(P_{3}^{(m)})=s.$ So, if $s>2$, $N_{\mathcal{H}_{r,s}^{e}}(P_{3}^{(m)})> N_{\mathcal{H}_{r+s,0}^{e}}(P_{3}^{(m)})$.

Therefore, if $r\geq s\geq 1$, we have
$N_{\mathcal{H}_{r,s}^{e}}(P_{3}^{(m)})> N_{\mathcal{H}_{r+s,0}^{e}}(P_{3}^{(m)}).$
\end{proof}
The following theorem gives the first hypergraph in an $S$-order of all $m$-uniform hypertrees.
\begin{thm}
In an $S$-order of $\textbf{T}_{q}$, the first hypergraph is the hyperpath $P_{q}^{(m)}$.
\end{thm}
\begin{proof}
In an $S$-order of $\textbf{T}_{q}$,  by Theorem \ref{zsy1}, the first hypergraph is in  $\textbf{T}$. When $m=2$, $\textbf{T}=\{P_{q}\}.$ Therefore, in an $S$-order of $\textbf{T}_{q}$, the first graph is the path $P_{q}$. When $m>2$,
since the spectral moments $\mathrm{S}_{0},\mathrm{S}_{1},\ldots,\mathrm{S}_{3m-1}$ are the same in $\textbf{T}$, the first significant spectral moment is the $3m$th. By Lemma \ref{w1}, $\mathrm{S}_{3m}$ is determined by the number of $S_{3}^{(m)}$ and $P_{3}^{(m)}$.

For any hypertree $\mathcal{T}$ in $\textbf{T}$, $N_{\mathcal{T}}(S_{3}^{(m)})=0$.
Let $e(\mathcal{T})$ denote the set of all edges of $\mathcal{T}$ that contain at least 3 vertices whose degree is equal to 2.
Fix a vertex $v$ of degree $2$ as a root. Let $\mathcal{T}_{1}, \mathcal{T}_{2}$ be the hypertrees attached at $v$.
We can repeatedly apply the transformation from Lemma \ref{z1} at any two vertices $u_{1}, u_{2}\in e\in e(\mathcal{T})$ with largest distance from the root in
every hypertree $\mathcal{T}_{i}$ and $d_{u_{1}}=d_{u_{2}}=2$, as long as $\mathcal{T}_{i}$ does not become a hyperpath. From Lemma \ref{z1}, each application of this transformation strictly decreases the number of  sub-hyperpaths with $3$ edges. In the end of this process, we arrive at the hyperpath $P_{q}^{(m)}$. Therefore, in an $S$-order of $\textbf{T}_{q}$, the first hypergraph is the hyperpath $P_{q}^{(m)}$.
\end{proof}
\section{The $S$-order in unicyclic hypergraphs}

In this section, the expressions of $2m$th and $3m$th order spectral moments of linear unicyclic $m$-uniform  hypergraphs are
obtained in terms of the number of sub-hypergraphs. We characterize the first and last hypergraphs in an $S$-order of  all linear unicyclic $m$-uniform
hypergraphs with given girth. And we give the last hypergraph in an $S$-order of  all linear unicyclic $m$-uniform  hypergraphs.




Let $\mathcal{H}(\omega)$ be a weighted uniform hypergraph, where $\omega: E(\mathcal{H})\rightarrow \mathbb{Z}^{+}$. Let $\omega(\mathcal{H})=\sum_{e\in E(\mathcal{H})}\omega(e)$ and $d_{v}(\mathcal{H}(\omega))=\sum_{e\in E_{v}(\mathcal{H})}\omega(e)$, where $E_{v}(\mathcal{H}):=\{e\in E(\mathcal{H})|v\in e\}$.  Let $C_{n}$ be a cycle with $n$ edges.
In \cite{ocestrada}, the formula for the spectral moments of linear unicyclic $m$-uniform  hypergraphs was given.
\begin{thm}\cite{ocestrada}\label{sp8}
Let $\mathcal{U}$ be a linear unicyclic $m$-uniform  hypergraph with girth $n$. If $m\mid d~(d\neq0)$ , then
\begin{equation}\label{888}
\mathrm{S}_{d}(\mathcal{U})=d(m-1)^{|V(\mathcal{U})|}(\sum\limits_{\mathcal{\widehat{T}}\in \mathcal{B}_{tree}(\mathcal{U})}tr_{d}(\mathcal{\widehat{T}})+\sum\limits_{\mathcal{G}\in \mathcal{B}_{cycle}(\mathcal{U})}tr_{d}(\mathcal{G}))
\end{equation}
and
$$
tr_{d}(\mathcal{\widehat{T}})=\sum\limits_{\omega:\omega(\mathcal{\widehat{T}})=d/m}(m-1)^{-|V(\mathcal{\widehat{T}})|}m^{(m-2)|E(\mathcal{\widehat{T}})|}\prod\limits_{v\in V(\mathcal{\widehat{T}})}(d_{v}(\mathcal{\widehat{T}}(\omega))-1)!\prod\limits_{e\in E(\mathcal{\widehat{T}})}\frac{\omega(e)^{m-1}}{(\omega(e)!)^{m}},
$$
$$
tr_{d}(\mathcal{G})=\sum\limits_{\omega:\omega(\mathcal{G})=d/m}2(m-1)^{-|V(\mathcal{G})|}m^{(m-2)|E(\mathcal{G})|-1}\prod\limits_{v\in V(\mathcal{G})}(d_{v}(\mathcal{G}(\omega))-1)!\prod\limits_{e\in E(\mathcal{G})}\frac{\omega(e)^{m-1}}{(\omega(e)!)^{m}}\Omega_{C_{n}^{(m)}(\omega^{0})},
$$
where $$\Omega_{C_{n}^{(m)}(\omega^{0})}=\sum\limits_{x=0}^{2\omega_{min}^{0}}\prod\limits_{i=1}^{n}\frac{(\omega_{i}^{0}!)^{2}}{(\omega_{i-1}^{0}+\omega_{min}^{0}-x)!(\omega_{i}^{0}-\omega_{min}^{0}+x)!}\sum\limits_{l=0}^{n-1}\prod\limits_{i=1}^{l}(\omega_{i}^{0}+\omega_{min}^{0}-x)
\prod\limits_{i=l+2}^{n}(\omega_{i}^{0}-\omega_{min}^{0}+x),
$$
$\omega_{min}^{0}=\mathrm{min}_{i\in n}\omega_{i}^{0}$, $\omega_{i}^{0}=\omega^{0}(e_{i}), i\in[n]$, $\mathcal{B}_{tree}(\mathcal{U})$ denotes the set of connected sub-hypergraphs of $\mathcal{U}$ which are hypertrees, $\mathcal{B}_{cycle}(\mathcal{U})$ denotes the set of connected sub-hypergraphs of $\mathcal{U}$ which contain the hypercycle.

If $m\nmid d$, then $\mathrm{S}_{d}(\mathcal{U})=0$.
\end{thm}

We give expressions of $2m$th and $3m$th order spectral moments of a linear unicyclic $m$-uniform  hypergraph in terms of the number of some sub-hypergraphs.
\begin{cor}\label{w3}
Let $\mathcal{U}$ be a linear unicyclic $m$-uniform  hypergraph. Then we have
$$
\mathrm{S}_{2m}(\mathcal{U})=m^{(m-1)}(m-1)^{|V(\mathcal{U})|-m}N_{\mathcal{U}}(P_{1}^{(m)})+2m^{2m-3}(m-1)^{|V(\mathcal{U})|-2m+1}N_{\mathcal{U}}(P_{2}^{(m)}).
$$
\end{cor}
\begin{proof}
Since $2m/m<g(\mathcal{U})$, the second summand in (\ref{888}) does not appear.
By Theorem \ref{sp8}, we have
\begin{align*}
\mathrm{S}_{2m}(\mathcal{U})&=2m(m-1)^{|V(\mathcal{U})|}\sum\limits_{\mathcal{\widehat{T}}\in \mathcal{B}_{tree}(\mathcal{U})}\sum\limits_{\omega:\omega(\mathcal{\widehat{T}})=2}(m-1)^{-|V(\mathcal{\widehat{T}})|}m^{(m-2)|E(\mathcal{\widehat{T}})|}\\
&\prod\limits_{v\in V(\mathcal{\widehat{T}})}(d_{v}(\mathcal{\widehat{T}}(\omega))-1)!\prod\limits_{e\in E(\mathcal{\widehat{T}})}\frac{\omega(e)^{m-1}}{(\omega(e)!)^{m}}.
\end{align*}
Since $\omega(\mathcal{\widehat{T}})=\sum_{e\in E(\mathcal{\widehat{T}})}\omega(e)=2$, $\widehat{T}$ is an edge $e$ with $\omega(e)=2$ or $\widehat{T}$ is a hyperpath of length $2$ with $\omega(e_{i})=1, i\in[2]$, where $E(\mathcal{\widehat{T}})=\{e_{1},e_{2}\}$.
So
\begin{align*}
\mathrm{S}_{2m}(\mathcal{U})&=2m(m-1)^{|V(\mathcal{U})|}((m-1)^{-m}m^{(m-2)}\frac{2^{m-1}}{2^{m}}N_{\mathcal{U}}(P_{1}^{(m)})+(m-1)^{1-2m}m^{2(m-2)}N_{\mathcal{U}}(P_{2}^{(m)}))\\
&=m^{(m-1)}(m-1)^{|V(\mathcal{U})|-m}N_{\mathcal{U}}(P_{1}^{(m)})+2m^{2m-3}(m-1)^{|V(\mathcal{U})|-2m+1}N_{\mathcal{U}}(P_{2}^{(m)}).
\end{align*}
\end{proof}
\begin{cor}\label{sp11}
Let $\mathcal{U}$ be a linear unicyclic $m$-uniform  hypergraph with girth $g$ $(g>3)$. Then we have
\begin{align*}
\mathrm{S}_{3m}(\mathcal{U})
&=(m-1)^{|V(\mathcal{U})|-m}m^{m-1}N_{\mathcal{U}}(P_{1}^{(m)})+6m^{2m-3}(m-1)^{|V(\mathcal{U})|+1-2m}N_{\mathcal{U}}(P_{2}^{(m)})\\
&+3m^{3m-5}(m-1)^{|V(\mathcal{U})|+2-3m}N_{\mathcal{U}}(P_{3}^{(m)})+6m^{3m-5}(m-1)^{|V(\mathcal{U})|+2-3m}N_{\mathcal{U}}(S_{3}^{(m)}).
\end{align*}
Let $\mathcal{U}$ be a linear unicyclic $m$-uniform  hypergraph with girth $3$. Then we have
\begin{align*}
\mathrm{S}_{3m}(\mathcal{U})
&=(m-1)^{|V(\mathcal{U})|-m}m^{m-1}N_{\mathcal{U}}(P_{1}^{(m)})+6m^{2m-3}(m-1)^{|V(\mathcal{U})|+1-2m}N_{\mathcal{U}}(P_{2}^{(m)})\\
&+3m^{3m-5}(m-1)^{|V(\mathcal{U})|+2-3m}N_{\mathcal{U}}(P_{3}^{(m)})+6m^{3m-5}(m-1)^{|V(\mathcal{U})|+2-3m}N_{\mathcal{U}}(S_{3}^{(m)})\\
&+24m^{3m-6}(m-1)^{|V(\mathcal{U})|-3m+3}.
\end{align*}
\end{cor}
\begin{proof}
When $g>3$, since $3m/m<g$, the second summand in (\ref{888}) does not appear.
By Theorem \ref{sp8}, we have
\begin{align*}
\mathrm{S}_{3m}(\mathcal{U})&=3m(m-1)^{|V(\mathcal{U})|}\sum\limits_{\mathcal{\widehat{T}}\in \mathcal{B}_{tree}(\mathcal{U})}\sum\limits_{\omega:\omega(\mathcal{\widehat{T}})=3}(m-1)^{-|V(\mathcal{\widehat{T}})|}m^{(m-2)|E(\mathcal{\widehat{T}})|}\\
&\prod\limits_{v\in V(\mathcal{\widehat{T}})}(d_{v}(\mathcal{\widehat{T}}(\omega))-1)!\prod\limits_{e\in E(\mathcal{\widehat{T}})}\frac{\omega(e)^{m-1}}{(\omega(e)!)^{m}}.
\end{align*}

Since $\omega(\mathcal{\widehat{T}})=\sum_{e\in E(\mathcal{\widehat{T}})}\omega(e)=3$, we have \\
(1). $\widehat{T}$ is an edge $e$ with $\omega(e)=3$; \\
(2). $\widehat{T}$ is a hyperpath of length $2$ with $\omega(e_{1})=1$, $\omega(e_{2})=2$ or $\omega(e_{1})=2$, $\omega(e_{2})=1$, where $E(\mathcal{\widehat{T}})=\{e_{1},e_{2}\}$; \\
(3). $\widehat{T}$ is a hyperpath of length $3$ with $\omega(e_{i})=1, i\in[3]$, where $E(\mathcal{\widehat{T}})=\{e_{1},e_{2},e_{3}\}$;\\
(4). $\widehat{T}$ is a hyperstar  with $3$ edges and $\omega(e_{i})=1, i\in[3]$, where $E(\mathcal{\widehat{T}})=\{e_{1},e_{2},e_{3}\}$.

Therefore,
\begin{align*}
\mathrm{S}_{3m}(\mathcal{U})&=3m(m-1)^{|V(\mathcal{U})|}((m-1)^{-m}m^{(m-2)}(2!)^{m}\frac{3^{m-1}}{(3!)^{m}}N_{\mathcal{U}}(P_{1}^{(m)})\\
&+(m-1)^{1-2m}m^{2(m-2)}2!\frac{2^{m-1}}{(2!)^{m}}2N_{\mathcal{U}}(P_{2}^{(m)})\\
&+(m-1)^{2-3m}m^{3(m-2)}N_{\mathcal{U}}(P_{3}^{(m)})+(m-1)^{2-3m}m^{3(m-2)}2!N_{\mathcal{U}}(S_{3}^{(m)}))\\
&=(m-1)^{|V(\mathcal{U})|-m}m^{m-1}N_{\mathcal{U}}(P_{1}^{(m)})+6m^{2m-3}(m-1)^{|V(\mathcal{U})|+1-2m}N_{\mathcal{U}}(P_{2}^{(m)})\\
&+3m^{3m-5}(m-1)^{|V(\mathcal{U})|+2-3m}N_{\mathcal{U}}(P_{3}^{(m)})+6m^{3m-5}(m-1)^{|V(\mathcal{U})|+2-3m}N_{\mathcal{U}}(S_{3}^{(m)}).
\end{align*}
When $g=3$,
since $\omega(\mathcal{\widehat{T}})=\sum_{e\in E(\mathcal{\widehat{T}})}\omega(e)=3$ , we have \\
(1). $\widehat{T}$ is an edge $e$ with $\omega(e)=3$;\\
(2). $\widehat{T}$ is a hyperpath of length $2$ with $\omega(e_{1})=1$, $\omega(e_{2})=2$ or $\omega(e_{1})=2$, $\omega(e_{2})=1$, where $E(\mathcal{\widehat{T}})=\{e_{1},e_{2}\}$;\\
(3). $\widehat{T}$ is a hyperpath of length $3$ with $\omega(e_{i})=1, i\in[3]$, where $E(\mathcal{\widehat{T}})=\{e_{1},e_{2},e_{3}\}$;\\
(4). $\widehat{T}$ is a hyperstar with $3$ edges and $\omega(e_{i})=1, i\in[3]$, where $E(\mathcal{\widehat{T}})=\{e_{1},e_{2},e_{3}\}$.

Since $\omega(\mathcal{G})=\sum_{e\in E(\mathcal{G})}\omega(e)=3$, $\mathcal{G}$ is a hypercycle with girth $3$, $\omega_{i}^{0}=\omega^{0}(e_{i})=1, i\in[3]$ and $\Omega_{C_{3}^{(m)}(\omega^{0})}=4$, where $E(\mathcal{G})=\{e_{1},e_{2},e_{3}\}$.
By Theorem \ref{sp8}, we have
\begin{align*}
\mathrm{S}_{3m}(\mathcal{U})&=3m(m-1)^{|V(\mathcal{U})|}((m-1)^{-m}m^{(m-2)}(2!)^{m}\frac{3^{m-1}}{(3!)^{m}}N_{\mathcal{U}}(P_{1}^{(m)})\\
&+(m-1)^{1-2m}m^{2(m-2)}2!\frac{2^{m-1}}{(2!)^{m}}2N_{\mathcal{U}}(P_{2}^{(m)})+(m-1)^{2-3m}m^{3(m-2)}N_{\mathcal{U}}(P_{3}^{(m)})\\
&+(m-1)^{2-3m}m^{3(m-2)}2!N_{\mathcal{U}}(S_{3}^{(m)})
+2(m-1)^{-3m+3}m^{3(m-2)-1}4) \\
&=(m-1)^{|V(\mathcal{U})|-m}m^{m-1}N_{\mathcal{U}}(P_{1}^{(m)})+6m^{2m-3}(m-1)^{|V(\mathcal{U})|+1-2m}N_{\mathcal{U}}(P_{2}^{(m)})\\
&+3m^{3m-5}(m-1)^{|V(\mathcal{U})|+2-3m}N_{\mathcal{U}}(P_{3}^{(m)})+6m^{3m-5}(m-1)^{|V(\mathcal{U})|+2-3m}N_{\mathcal{U}}(S_{3}^{(m)})\\
&+24m^{3m-6}(m-1)^{|V(\mathcal{U})|-3m+3}.
\end{align*}
\end{proof}

The set of all linear unicyclic $m$-uniform  hypergraphs with $e+f$ edges which contain a hypercycle $C_{e}^{(m)}$ will be denoted by $\textbf{U}^{m}_{ef}$.
Let $F^{(m)}_{ef}$ be the linear unicyclic $m$-uniform  hypergraph obtained from the hypercycle $C^{(m)}_{e}$ by
attached $f$ pendant edges to one of non core vertices on $C^{(m)}_{e}$.
The following theorem gives the last hypergraph in an $S$-order of all linear unicyclic $m$-uniform
hypergraphs with given girth.

\begin{thm}\label{sp9}
In an $S$-order of $\textbf{U}^{m}_{ef}$ the last hypergraph is $F^{(m)}_{ef}$.
\end{thm}
\begin{proof}
Since in $\textbf{U}^{m}_{ef}$ the spectral moments $\mathrm{S}_{0},\mathrm{S}_{1},\ldots,\mathrm{S}_{2m-1}$ are the same, the first significant spectral moment is the $2m$th. By Corollary \ref{w3}, $\mathrm{S}_{2m}$ is determined by the number of $P_{2}^{(m)}$. The number of vertices of linear unicyclic $m$-uniform  hypergraphs with $e+f$ edges is $(e+f)(m-1)$. For any $\mathcal{U}\in \textbf{U}^{m}_{ef}$, we have
$$
N_{\mathcal{U}}(P_{2}^{(m)})=\sum\limits_{i=1}^{em+fm-e-f}{d_{i}\choose2}=\frac{1}{2}\sum\limits_{i=1}^{em+fm-e-f}d_{i}^{2}-\frac{em+fm}{2}=\frac{1}{2}M(\mathcal{U})-\frac{em+fm}{2},
$$
where $d_{1}+d_{2}+\cdots+d_{em+fm-e-f}=em+fm$.

Repeating transformation 1, any linear unicyclic $m$-uniform  hypergraph in $\textbf{U}^{m}_{ef}$ can be changed into a linear unicyclic $m$-uniform  hypergraph such that all the edges not on $C^{(m)}_{e}$ are pendant edges and incident with non core vertices of $C^{(m)}_{e}$.

After repeating transformation 1, if we repeat transformation 2, any linear unicyclic $m$-uniform  hypergraph in $\textbf{U}^{m}_{ef}$
can be changed into a linear unicyclic $m$-uniform  hypergraph obtained from the hypercycle $C^{(m)}_{e}$ by
attached $f$ pendant edges to one of non core vertices on $C^{(m)}_{e}$.

%
From Lemma \ref{sp1} and Lemma \ref{sp2},
each application of transformation 1 or 2 strictly increases the Zagreb index.
Hence, in an $S$-order of $\textbf{U}^{m}_{ef}$ the last hypergraph is $F^{(m)}_{ef}$.

\end{proof}
The set of all linear unicyclic $m$-uniform  hypergraphs with $q$ edges will be denoted by $\textbf{U}_{q}$. The following theorem gives the last hypergraph in an $S$-order of all linear unicyclic $m$-uniform
hypergraphs.
\begin{thm}
In an $S$-order of $\textbf{U}_{q}$ the last hypergraph is $F^{(m)}_{3(q-3)}$.
\end{thm}
\begin{proof}
By Theorem \ref{sp9}, we get that in an $S$-order of $\textbf{U}^{m}_{l(q-l)}$ the last hypergraph is $F^{(m)}_{l(q-l)}$.
By the definition of the Zagreb index, we have
$M(F_{l(q-l)}^{(m)})=(m-2)l+(q-l)(m-1)+4(l-1)+(q-l+2)^{2}=l^{2}-l-2ql+qm+3q+q^{2}, 3\leq l\leq q$. Since the derivative of $M(F_{l(q-l)}^{(m)})$ over $l$ is equal to $2l-1-2q<0$, $M(F_{l(q-l)}^{(m)})\leq M(F^{(m)}_{3(q-3)})$ for $3\leq l\leq q$ with the equality if and only if $l=3$. Hence, in an $S$-order of $\textbf{U}_{q}$ the last hypergraph is $F^{(m)}_{3(q-3)}$.
\end{proof}

For $m\geq3,$ let $\textbf{U}$ be the set of all linear unicyclic $m$-uniform  hypergraphs with $e+f$ edges and girth $e$ such that the degree of all the vertices is less than or equal to $2$.
We characterize the first few hypergraphs
in the $S$-order of all linear unicyclic $m$-uniform
hypergraphs with given girth.

\begin{thm}\label{sp10}
For $m\geq3,$
$$
\textbf{U}\prec_{s}\textbf{U}^{m}_{ef}\setminus \textbf{U}.
$$
\end{thm}
\begin{proof}

As in the proof of Theorem \ref{sp9} we pay attention to the Zagreb index.
Repeating transformation 3, any $m$-uniform hypertree attached to an $m$-uniform hypergraph $\mathcal{H}$ can be changed into a binary $m$-uniform hypertree.
After repeating transformation 3, if we repeat transformation 4, then any linear unicyclic $m$-uniform  hypergraph in $\textbf{U}^{m}_{ef}$ can be changed into a linear unicyclic $m$-uniform  hypergraph in $\textbf{U}$.
And from Lemma \ref{sp4} and Lemma \ref{sp5}, the Zagreb indices decrease.
Hence, we have $\textbf{U}\prec_{s}\textbf{U}^{m}_{ef}\setminus \textbf{U}.$
\end{proof}
We give a
transformation which will decrease the number of sub-hyperpaths with $3$ edges of hypergraphs as follows:

$\textbf{Transformation 5}$: Let $\mathcal{P}_{i}\neq P_{0}^{(m)}$ be an $m$-uniform hyperpath, $u_{i}$ be a pendent vertex of $\mathcal{P}_{i}$ for each $i\in[p]$ and $v_{1},v_{2},\ldots,v_{e(m-2)}$ be core vertices of a linear $m$-uniform hypercycle $C^{(m)}_{e}$, where $m\geq 3$ and $2\leq p\leq e(m-2)$. Let $\mathcal{H}_{1}=C^{(m)}_{e}(v_{1},\ldots,v_{p})\bigodot(\mathcal{P}_{1}(u_{1}),\ldots,\mathcal{P}_{p}(u_{p}))$.
Suppose that  $u_{1}\in e_{1}$ in $\mathcal{P}_{1}$, $w_{1}\in V(\mathcal{P}_{2})$ is a pendent vertex of $\mathcal{H}_{1}$,
let $\mathcal{H}_{2}$ be obtained from $\mathcal{H}_{1}$ by deleting $e_{1}$ and adding $(e_{1}\setminus\{u_{1}\})\bigcup\{w_{1}\}$.
\begin{lem}\label{zsy2}
Let $\mathcal{H}_{2}$ be obtained from $\mathcal{H}_{1}$ by transformation 5. Then $N_{\mathcal{H}_{2}}(P_{3}^{(m)})<N_{\mathcal{H}_{1}}(P_{3}^{(m)})$.
\end{lem}
\begin{proof}

Let $\mathcal{H}_{3}=C^{(m)}_{e}(v_{2},\ldots,v_{p})\bigodot(\mathcal{P}_{2}(u_{2}),\ldots,\mathcal{P}_{p}(u_{p}))$ and $\mathcal{P}_{1}'=\mathcal{P}_{1}-e_{1}+(e_{1}\setminus\{u_{1}\})\bigcup\{w_{1}\}$.
So $P_{3}(\mathcal{H}_{1})=P_{3}(\mathcal{H}_{3})+P_{3}(\mathcal{P}_{1})+P_{\mathcal{H}_{1}}$ and $P_{3}(\mathcal{H}_{2})=P_{3}(\mathcal{H}_{3})+P_{3}(\mathcal{P}_{1}' )+P_{\mathcal{H}_{2}}$,
where $P_{\mathcal{H}_{1}}~(P_{\mathcal{H}_{2}})$ is the set of all the sub-hyperpaths with $3$ edges of $\mathcal{H}_{1}(\mathcal{H}_{2})$, each of them contains both at least one edge in $E(\mathcal{H}_{3})$ and at least one edge in $E(\mathcal{P}_{1})~(E(\mathcal{P}_{1}')).$ We have $|E(\mathcal{P}_{1})|=|E(\mathcal{P}_{1}')|$ and $N_{\mathcal{P}_{1}'}(P_{3}^{(m)})=N_{\mathcal{P}_{1}}(P_{3}^{(m)}) $.

If  $|E(\mathcal{P}_{1})|=1$, since $p\geq2$, in $P_{\mathcal{H}_{1}}$
there are 2 hyperpaths  at least which contain $e_{1}$ and two edges in $E(\mathcal{H}_{3})$. In $P_{\mathcal{H}_{2}}$ there is a  hyperpath which contain $(e_{1}\setminus\{u_{1}\})\bigcup\{w_{1}\}$ and two edges in $E(\mathcal{H}_{3})$.
Therefore, we have $|P_{\mathcal{H}_{1}}|-|P_{\mathcal{H}_{2}}|\geq1$. Hence, $N_{\mathcal{H}_{1}}(P_{3}^{(m)})-N_{\mathcal{H}_{2}}(P_{3}^{(m)})\geq1$. So, $N_{\mathcal{H}_{2}}(P_{3}^{(m)})<N_{\mathcal{H}_{1}}(P_{3}^{(m)})$.

If $|E(\mathcal{P}_{1})|\geq2$, since $p\geq2$,
 in $P_{\mathcal{H}_{1}}$ there are 2 hyperpaths  at least which contain $e_{1}$ and two edges in $E(\mathcal{H}_{3})$ and there is a hyperpath which contain two edges in $E(\mathcal{P}_{1})$ and an edge in $E(\mathcal{H}_{3})$.
In $P_{\mathcal{H}_{2}}$ there is a  hyperpath which contain $(e_{1}\setminus\{u_{1}\})\bigcup\{w_{1}\}$ and two edges in $E(\mathcal{H}_{3})$ and there is a hyperpath which contain two edges in $E(\mathcal{P}_{1}')$ and an edge in $E(\mathcal{H}_{3})$.
Therefore, we have $|P_{\mathcal{H}_{1}}|-|P_{\mathcal{H}_{2}}|\geq1$. Hence, $N_{\mathcal{H}_{1}}(P_{3}^{(m)})-N_{\mathcal{H}_{2}}(P_{3}^{(m)})\geq1$. So, $N_{\mathcal{H}_{2}}(P_{3}^{(m)})<N_{\mathcal{H}_{1}}(P_{3}^{(m)})$.

\end{proof}


Let $E_{ef}^{m}$ be the linear unicyclic $m$-uniform  hypergraph obtained by the coalescence of $C^{(m)}_{e}$ at one of its core vertices with $P^{(m)}_{f}$ at one of its pendent vertices.
The following theorem gives the first hypergraph in an $S$-order of all linear unicyclic $m$-uniform
hypergraphs with given girth.
\begin{thm}
For $m\geq 3$, in an $S$-order of $\textbf{U}^{m}_{ef}$ the first hypergraph is $E_{ef}^{m}$.
\end{thm}
\begin{proof}


In an $S$-order of $\textbf{U}^{m}_{ef}$,  by Theorem \ref{sp10}, the first hypergraph is in  $\textbf{U}$.
Since the spectral moments $\mathrm{S}_{0},\mathrm{S}_{1},\ldots,\mathrm{S}_{3m-1}$ are the same in $\textbf{U}$, the first significant spectral moment is the $3m$th. By Corollary \ref{sp11}, $\mathrm{S}_{3m}$ is determined by the number of $S_{3}^{(m)}$ and $P_{3}^{(m)}$. For any $\mathcal{H}\in \textbf{U}$, $N_{\mathcal{H}}(S_{3}^{(m)})=0$.

Let $\mathcal{T}_{1},\ldots,\mathcal{T}_{p}$ be pairwise disjoint binary $m$-uniform hypertrees, $u_{i}$ be a pendent vertex of $\mathcal{T}_{i}$ for each $i\in[p]$ and $v_{1},\ldots,v_{p}$ be core vertices of $C^{(m)}_{e}$, where $1\leq p\leq e(m-2)$ and $\sum_{i=1}^{p}|E(\mathcal{T}_{i})|=f$.
For any $\mathcal{H}=C^{(m)}_{e}(v_{1},\ldots,v_{p})\bigodot(\mathcal{T}_{1}(u_{1}),\ldots,\mathcal{T}_{p}(u_{p}))\in \textbf{U}$,
 let $e(\mathcal{H})$ denote the set of all edges of $\mathcal{H}-E(C^{(m)}_{e})$ that contain at least 3 vertices whose degree is equal to 2.
 Let the vertex $u_{i}$ as a root in $\mathcal{T}_{i}$.
We can repeatedly apply the transformation from Lemma \ref{z1} at any two vertices $u, v\in e\in e(\mathcal{H})$ with largest distance from the root in every hypertree
$\mathcal{T}_{i}$ and $d_{u}=d_{v}=2$, as long as $\mathcal{T}_{i}$ does not become a hyperpath. By Lemma \ref{z1}, each application of this transformation strictly decreases the number of sub-hyperpaths with $3$ edges.


When all hypertrees $\mathcal{T}_{1},\ldots,\mathcal{T}_{p}$  turn into hyperpaths, we can repeatedly apply the transformation 5,
as long as there exist at least two hyperpaths of length at least one,
By Lemma \ref{zsy2}, each application of transformation 5 strictly decreases the number of sub-hyperpaths with $3$ edges.
In the end of this process, we arrive at the $E_{ef}^{m}$.
\end{proof}

\vspace{3mm}

\noindent
\textbf{Acknowledgments}
\vspace{3mm}
\noindent

This work is supported by the National Natural Science Foundation of China (No.
11801115, No. 12071097, No. 12042103 and No. 12242105), the Natural Science Foundation of the
Heilongjiang Province (No. QC2018002) and the Fundamental Research Funds for
the Central Universities.

\section*{References}
\bibliographystyle{unsrt}
\bibliography{pbib}
\end{spacing}
\end{document}